\documentclass{siamltex}

\usepackage{latexsym,amssymb,amsmath,enumerate,bbm,graphicx,psfrag,color,url}
\usepackage[latin1]{inputenc}

\definecolor{mygray}{rgb}{0.3333,0.3333,0.3333} %xfig: #555555

\def\norm#1{\hspace{0.2ex} \|#1\| \hspace{0.2ex}} 

\newcommand{\labeq}[1]{\label{eq:#1}}			% Nummerierung von Gleichung
\def\req#1{{\rm(\ref{eq:#1})}}

\newcommand{\R}{\ensuremath{\mathbbm{R}}} 
\newcommand{\N}{\ensuremath{\mathbbm{N}}} 
\newcommand{\range}{\mathcal R}

\newcommand{\dx}[1][x]{\ensuremath{\,{\rm{d}} #1}}
\newcommand{\DS}{\displaystyle} 
\newcommand{\TS}{\textstyle} 
\newcommand{\Ld}{L_\diamond}
\newcommand{\Hd}{H_\diamond}
\newcommand{\Langle}{\left\langle}
\newcommand{\Rangle}{\right\rangle}

\newcommand{\outsupp}{\mathrm{out}_{\partial\Omega}\,\mathrm{supp}\,}
\newcommand{\innsupp}{\mathrm{inn\,supp}\,}
\newcommand{\out}{\mathrm{out}_{\partial\Omega}\,}
\newcommand{\inn}{\mathrm{int}\,}

\newcommand{\kommentar}[1]{}

\newtheorem{example}[theorem]{Example}
\newtheorem{remark}[theorem]{Remark}

\begin{document}

\title{Monotonicity based shape reconstruction in electrical impedance tomography}
       
\author{Bastian Harrach\footnotemark[2] and Marcel Ullrich\footnotemark[3]}
\renewcommand{\thefootnote}{\fnsymbol{footnote}}

\footnotetext[2]{Birth name: Bastian Gebauer, Department of Mathematics, University of Stuttgart, Germany\\
{\tt harrach@math.uni-stuttgart.de}}
\footnotetext[3]{Department of Mathematics, University of Stuttgart, Germany\\
{\tt marcel.ullrich@mathematik.uni-stuttgart.de}}
 
\footnotetext{\hrule \vspace{1ex} \centering This is a preprint version of a journal article published in\\
\emph{SIAM J.\ Math.\ Anal.} \textbf{45}(6), 3382--3403, 2013
(\url{http://dx.doi.org/10.1137/120886984}).
}

\maketitle

\pagestyle{myheadings}
\thispagestyle{plain}
\markboth{Bastian Harrach and Marcel Ullrich}{Monotonicity based shape reconstruction in EIT}

\begin{abstract}
Current-voltage measurements in electrical impedance tomography can be 
partially ordered with respect to definiteness of the associated self-adjoint Neumann-to-Dirichlet operators (NtD).
With this ordering, a point-wise larger conductivity leads to smaller current-voltage measurements, and smaller conductivities lead to larger measurements.

We present a converse of this simple monotonicity relation and use it to 
solve the shape reconstruction (aka inclusion detection) problem in EIT. 
The outer shape of a region where the conductivity differs from a known background conductivity
can be found by simply comparing the measurements to that of smaller or larger test regions.
\end{abstract}

%%%%%%%%%%%%%%%%%%%%%%%%%%%%%%%%%%%%%%%%%%%%%%%%%%%%%%%%%%%%%%%%%%%%%%%%%%%%
\section{Introduction}
\label{Sec:intro}
%%%%%%%%%%%%%%%%%%%%%%%%%%%%%%%%%%%%%%%%%%%%%%%%%%%%%%%%%%%%%%%%%%%%%%%%%%%%

We consider the shape reconstruction (aka inclusion detection) problem in electrical impedance tomography (EIT).
Let $\Omega$ describe an electrically conducting object which contains inclusions in which the conductivity 
$\sigma(x)$ differs from an otherwise known background conductivity.
Our aim is to detect these inclusions from current/voltage-measurements on the boundary $\partial \Omega$.

We assume that $\Omega\subset \R^n$, $n\geq 2$ is a domain with smooth boundary $\partial \Omega$ and outer normal vector $\nu$. For ease of presentation we also assume that $\Omega$ is bounded, the background conductivity is equal to $1$ and that we are given measurements on the complete boundary $\partial \Omega$. Our results easily extend to inhomogeneous (but known) backgrounds and partial boundary measurements, cf.\ section \ref{subsect:extensions}.

With these assumptions, our goal is to determine the inclusions shape, i.e., the set  $\supp  (\sigma-1)$, from knowledge of the Neumann-to-Dirichlet (NtD) operator  
%notation: introduce \kappa or write \sigma-1?
\[
\Lambda(\sigma):\ g\mapsto u^g_\sigma|_{\partial\Omega},
\]
where $u^g_\sigma$ is the solution of
\begin{equation*}
\nabla\cdot\sigma\nabla u^g_\sigma = 0 \text{ in }\Omega,\quad\sigma\partial_\nu u_\sigma^g\vert_{\partial\Omega} = g \text{ on }\partial\Omega,
\end{equation*}
cf.\ section \ref{subsect:NtD} for the precise mathematical setting.

In this work, we show that $\supp  (\sigma-1)$ can be reconstructed by
so-called \emph{monotonicity tests}, which
simply compare $\Lambda(\sigma)$ (in the sense of quadratic forms)
to NtD-operators $\Lambda(\tau)$ of test conductivities $\tau$. To be more precise,
the support of $\sigma-1$ can be reconstructed under the assumption that $\supp  (\sigma-1)\subset \Omega$
has connected complement. Otherwise, what we can reconstruct is essentially the support together with all holes that have no connection to the boundary $\partial \Omega$.

Moreover, we show 
that the test NtDs $\Lambda(\tau)$ can be replaced (without losing any information) by their linear approximations using the Fr\'echet derivative $\Lambda'(1)$ of $\Lambda(\sigma)$ around the background conductivity. Let us stress that the linearized tests still exactly recover the inclusion which is in accordance with the general principle that the linearized EIT problem still contains the exact shape information, cf. \cite{HS10}.

The term \emph{monotonicity tests} is used because our test criteria are motivated and partly follow from the simple and well-known monotonicity relation 
\begin{equation}\labeq{intro_monotonicity}
\sigma\leq \tau  \quad \mbox{ implies } \quad \Lambda(\sigma)\geq \Lambda(\tau).
\end{equation}
It seems quite natural and intuitive to probe the domain with
test inclusions using the implication \req{intro_monotonicity}, and this idea has been worked out and numerically tested in the works of 
Tamburrino and Rubinacci \cite{Tamburrino02,Tamburrino06}.
The main new part of this work is to rigorously justify this natural idea by proving a non-trivial converse of the implication \req{intro_monotonicity}.
Our proofs are based on the theory of localized potentials \cite{Geb08}.
%, and to replace 
%the test NtDs $\Lambda(\tau)$ (which require the computationally expensive solution of an inhomogenous forward problem) by their linearized counterparts without losing shape information.

For a quick impression of our result let us state it for two frequently considered special cases (see examples 
\ref{ex:main_th1}, \ref{ex:main_th1_linearized}, \ref{ex:main_th_disjoint} and \ref{ex:main_th_disjoint_linearized}). (Note that throughout the paper we use the relation symbol ''$\subset$'' instead of ''$\subseteq$'', if non-equality of the two related sets is obvious.)
\begin{enumerate}[(a)]
\item Let $\sigma=1+\chi_D$ where $D$ is open, and $\overline D\subset \Omega$ has a connected complement. Then for every open ball $B\subseteq \Omega$
\begin{align*}
B\subseteq D & \quad \mbox{ if and only if } \quad \Lambda(1+\chi_B)\geq \Lambda(\sigma)\\
& \quad \mbox{ if and only if } \quad \Lambda(1)+{\textstyle \frac{1}{2}} \Lambda'(1) \chi_B\geq \Lambda(\sigma).
\end{align*}
\item Let $\sigma=1+\chi_{D^+}-\frac{1}{2}\chi_{D^-}$ where $D^+,D^-\subseteq \Omega$ are open, 
$\overline{D^+}\cap \overline{D^-}=\emptyset$, and $\overline{D^+}\cup \overline{D^-}\subset \Omega$ has a connected complement. Then for every closed $C\subset \Omega$ with connected complement
\begin{align*}
D^+ \cup D^-\subseteq C \quad & \mbox{ if and only if } \quad \Lambda(1+\chi_C)\leq \Lambda(\sigma)\leq \Lambda(1-{\TS \frac{1}{2}}\chi_C)\\
& \mbox{ if and only if } \quad \Lambda(1)+ \Lambda'(1)\chi_C\leq \Lambda(\sigma)\leq \Lambda(1)- \Lambda'(1)\chi_C.
\end{align*}
\end{enumerate}
(a) is a special case of the \emph{definite case} in which either all inclusions have a higher conductivity, or all inclusions have a lower conductivity, than the background.
(a) shows how to test whether a small ball $B$ lies inside the inclusion or not. The inclusion can thus be obtained as the union of all balls that fulfill the test.

(b) is a special case of the more general \emph{indefinite case} in which the conductivity 
may differ in both directions from the background. Using the result in (b) we can test whether a large set $C$ contains the inclusions or not. The inclusion can thus be obtained as the intersection of all these large sets.

Our results show that (under quite general assumptions) 
monotonicity tests determine $\supp(\sigma-1)$ up to holes that have no connection to the boundary $\partial \Omega$. %When $\sigma$ is piecewise continuous and $\supp  (\sigma-1)\subset \Omega$
%has a connected complement, then $\supp(\sigma-1)$ is determined, cf. Corollary \ref{cor:exakt}.

Non-iterative methods for shape reconstruction problems have been studied intensively in the last 25 years, cf., e.g., the overview of Potthast \cite{potthast2006survey}. In the context of EIT, the inclusion detection problem was first considered by Friedmann and Isakov \cite{friedman1987detection,friedman1989uniqueness}. 
For the following brief overview, we restrict ourselves to the two most prominent and elaborated 
methods for detecting inclusions of unknown conductivity from the full Neumann-to-Dirichlet (or Dirichlet-to-Neumann) operator on all or part of the boundary: the Factorization Method and the Enclosure Method.

The Factorization Method (FM) was introduced by Kirsch \cite{Kir98,kirsch2000new} for inverse scattering problems and extended to impedance tomography by Brühl and Hanke \cite{Bru00,Bru01}. For its
further developments in the context of EIT see  \cite{Kir02,Han03,Geb05,Han07,Hyv07,Lec07,Nac07_2,GH07,Kirsch_book07,hakula2009computation,HS09,Schmitt2009,harrach2010factorization,schmitt2011factorization} and the recent review \cite{harrach2012recent}. 
The Factorization Method reconstructs the shape of inclusions (up to holes that have no connection to the boundary), but two major problems have not been solved so far. First of 
all, the method relies on a range test (or infinity test) for which there is no known convergent implementation (see, however, Lechleiter \cite{lechleiter2006regularization} for a first step in this direction).
Second, the method has only been justified for the \emph{definite case}
(or that the domain can be split into two a-priori known regions with the definiteness property, cf. Schmitt \cite{Schmitt2009} and the review \cite{harrach2012recent}).
% that
%the inclusions' conductivity is either everywhere higher or everywhere lower than the background

The Enclosure Method was introduced by Ikehata \cite{ikehata1999draw,ikehata2000reconstruction}. Further extensions including the use of the Sylvester-Uhlmann complex geometrical optics solutions have been worked out
in \cite{Bru00,ikehata2000numerical,ikehata2002regularized,ikehata2004electrical,Ide07,uhlmann2008reconstructing,ide2010local}. 
The method yields a stable testing criterion and it does not require
the definiteness assumption (see \cite{Ide07}). However, it does require the
construction of special, strongly oscillating probe functions and 
only reconstructs the convex hull of the inclusions (plus some non-convex features depending on the probe functions).

The herein presented monotonicity tests seem to be a particularly simple and intuitively appealing solution to the long-studied inclusion detection problem. They characterize the outer shape of the inclusions and not just the convex hull. They work for the general indefinite case (though the implementation is simpler in the definite case). Also, they allow a stable implementation (see remark~\ref{rem:convergence}),
and their linearized versions do not require solving inhomogeneous forward problems.

The paper is organized as follows. Section \ref{sec:basics_notations} introduces the mathematical setting
and the concept of inner and outer support. In section \ref{sect:monotonicity_locpot} we derive the main theoretical tools for our proofs: monotonicity estimates
and localized potentials. Section \ref{sect:reconstruction} then contains our main results: the characterization of inclusion 
by simple and stable monotonicity tests.

\section{Basic notations and support definitions}\label{sec:basics_notations}

\subsection{Basic notations and the mathematical setting}\label{subsect:NtD}

Let $\Omega\subset\R^n$, $n\geq 2$ be a bounded domain with smooth boundary $\partial\Omega$ and outer normal vector $\nu$. $L_+^\infty(\Omega)$ denotes the subspace of $L^\infty(\Omega)$-functions with positive essential infima. $\Hd^1(\Omega)$ and $\Ld^2(\partial \Omega)$ denote the spaces of $H^1$- and $L^2$-functions with vanishing integral mean on $\partial \Omega$.

The $L^2_\diamond(\partial\Omega)$-inner product is denoted by 
$\langle\cdot,\cdot\rangle$. For two bounded selfadjoint operators
$A,B:\ \Ld^2(\partial \Omega)\to \Ld^2(\partial \Omega)$ we write
\[
A\geq B
\]
if it holds in the sense of quadratic forms, i.e.,
\[
\Langle g,(A-B)g\Rangle\geq 0,\quad \mbox{ for all } g\in L^2_\diamond(\partial\Omega).
\]
For $\sigma_1,\sigma_2\in L^\infty(\Omega)$ we write $\sigma_1\geq \sigma_2$ if it holds pointwise (a.e.) on $\Omega$.

For $\sigma\in L^\infty_+(\Omega)$, the Neumann-to-Dirichlet (NtD) operator $\Lambda(\sigma)$ is defined by 
%notation: introduce \kappa or write \sigma-1?
\[
\Lambda(\sigma):\ L^2_\diamond(\partial \Omega)\to L^2_\diamond(\partial \Omega), \quad g\mapsto u^g_\sigma|_{\partial\Omega},
\]
where $u^g_\sigma\in H^1_\diamond(\Omega)$ is the unique solution of
\begin{equation}\labeq{math_model}
\nabla\cdot\sigma\nabla u^g_\sigma = 0 \text{ in }\Omega,\quad\sigma\partial_\nu u_\sigma^g\vert_{\partial\Omega} = g \text{ on }\partial\Omega
\end{equation}
which is equivalent to 
\begin{equation}\label{eq:math_model_variationel}
 \int_\Omega\sigma\nabla u^g_\sigma\cdot\nabla v\,\mathrm{d}x=\int_{\partial\Omega}gv\vert_{\partial\Omega}\,\mathrm{d}s\quad \mbox{ for all } v\in H^1_\diamond(\Omega).
\end{equation}
 It is well known $\Lambda(\sigma)$ is 
a selfadjoint compact linear operator, and that the associated bilinear form is given by
\[
\langle g, \Lambda(\sigma) h\rangle %:= \int_{\partial \Omega} g \Lambda(\sigma) h\dx[s]
= \int_\Omega \sigma \nabla u^g_\sigma \cdot \nabla u^h_\sigma \dx.
\]
%where, here and in the following,  $\langle \cdot , \cdot \rangle$ denotes the $\Ld^2(\partial \Omega)$-inner product. 

$\Lambda$ is Fr\'echet-differentiable, cf., e.g. Lechleiter
and Rieder \cite{Lec08} for a recent proof that uses only the abstract variational formulation (see also \cite{ikehata1990inversion} for similar results). Given some direction $\kappa\in L^\infty(\Omega)$ the derivative 
\[
\Lambda'(\sigma)\kappa:\ \Ld^2(\partial \Omega)\to \Ld^2(\partial \Omega)
\]
is the selfadjoint compact linear operator associated to the bilinear form
\begin{equation*}\labeq{Frechet_bilinearForm}
\langle \left( \Lambda'(\sigma)\kappa\right) g,h\rangle 
= -\int_\Omega \kappa \nabla u^g_\sigma
\cdot \nabla u^h_\sigma \dx.
\end{equation*}
Note that for $\kappa_1,\kappa_2\in L^\infty(\Omega)$ we obviously have that
\begin{equation}\labeq{Frechet_monotonicity}
\kappa_1\leq\kappa_2 \quad \mbox{ implies } \quad \Lambda'(\sigma)\kappa_1 \geq \Lambda'(\sigma)\kappa_2.
\end{equation}

The terms piecewise continuous and piecewise analytic are understood in the following sense.
\begin{definition}
\label{def:piecewise}
\begin{enumerate}[(a)]
\item A subset $\Gamma\subseteq \partial O$ of the boundary of an open set $O\subseteq \R^n$ is called a \emph{smooth boundary piece} if
it is a $C^\infty$-surface and $O$ lies on one side of it, i.e., if for each $z\in \Gamma$
there exists a ball $B_\epsilon(z)$ and a function $\gamma\in C^\infty(\R^{n-1},\R)$ such that upon relabeling and reorienting
\begin{align*}
\Gamma=\partial O\cap B_\epsilon(z)&=\{ x\in B_\epsilon(z) \; | \; x_n=\gamma(x_1,\ldots,x_{n-1}) \},\\
O\cap B_\epsilon(z)&=\{ x\in B_\epsilon(z) \; | \; x_n>\gamma(x_1,\ldots,x_{n-1}) \}.
\end{align*}
\item $O$ is said to have \emph{smooth boundary} if $\partial O$ is a union of smooth boundary pieces.
$O$ is said to have \emph{piecewise smooth boundary} if $\partial O$ is a countable union of the closures of smooth boundary pieces.
\item A function $\kappa \in L^\infty(\Omega)$ is called \emph{piecewise analytic} if there exist
finitely many pairwise disjoint subdomains $O_1,\ldots,O_M\subset \Omega$ with piecewise
smooth boundaries, such that $\overline{\Omega}= \overline{O_1\cup \ldots \cup O_M}$,
and $\kappa|_{O_m}$ has an extension which is (real-)analytic in a neighborhood of $\overline{O_m}$,
$m=1,\ldots,M$.
\item A function $\kappa\in L^\infty(\Omega)$ is called \emph{piecewise continuous} if
$\kappa$ is continuous on an open set $O\subset \Omega$ and $\Omega\setminus O$ is a set of zero measure.
\end{enumerate}
\end{definition}

% \begin{remark}[will be shifted to the proofs]
% A smooth boundary piece $\Gamma$ is an open and neither empty nor dense subset of the $n-1$-dimensional mannifold
% \[
% \{ x\in \R^n \; | \; x_n=\gamma(x_1,\ldots,x_{n-1}) \}
% \]
% so that its boundary $\overline{\Gamma} \setminus \Gamma$ is $n-2$-dimensional (Ch. IV, \S 4 im Hurewicz-Wallman). 
% The countable union of closed $n-2$-dimensional sets is  again $n-2$-dimensional (Thm. 3.2, \S 4 im Hurewicz-Wallman).
% Taking away a $n-2$-dimensional subset of a connected open set in $\R^n$ cannot make the set disconnected.
% (Ch. IV, \S 5 im Hurewicz-Wallman)
% \end{remark}

\subsection{Inner and outer support}\label{subsect:supportdef}

We will show that our method reconstructs $\supp(\sigma-1)$ (the \emph{inclusion}) up to holes that cannot be connected to the boundary $\partial \Omega$ without crossing the support. For the precise formulation, we will now introduce the concept of the \emph{inner} and the \emph{outer support} of a measurable function. For the frequently considered case that the inclusion has a connected complement
and the conductivity is piecewise continuous, the inner and the outer support only differ by the boundary of the support, cf.\ corollary~\ref{cor:exakt}.
The following has been inspired by the use of the infinity support of Kusiak and Sylvester \cite{Kus03}, cf.\ also \cite{GH08,HS10}.

\begin{definition}\label{def:connected}
A relatively open set $U\subseteq \overline{\Omega}$ is called \emph{connected to $\partial \Omega$} if $U\cap \Omega$ is connected and $U\cap \partial \Omega\neq \emptyset$.
\end{definition}
%Note that if $U\subseteq \overline{\Omega}$ is relatively open and connected to $\Omega$ then $U\cap \partial \Omega$ contains a smooth boundary piece
%and $U$ is path-connected.

\begin{definition}\label{def:support}
For a measurable function $\kappa:\ \Omega \to \R$ we define:
\begin{enumerate}[(a)]
\item the \emph{support} $\mathrm{supp}\,\kappa$ as the complement (in $\overline\Omega$) of the union of those relatively open  $U\subseteq\overline{\Omega}$, for which $\kappa|_U\equiv 0$,
\item the \emph{inner support} $\innsupp \kappa$ as the union of those open sets $U\subseteq \Omega$,
for which $\mathrm{ess\,inf}_{x\in U}|\kappa(x)|>0$.
\item the \emph{outer support} $\outsupp \kappa$ as the complement (in $\overline\Omega$) of the union of those relatively open $U\subseteq \overline{\Omega}$
that are connected to $\partial \Omega$ and for which $\kappa|_U\equiv 0$. 
\end{enumerate}
The interior of a set $M\subseteq\Omega$ is denoted by $\inn M$ and its closure (with respect to $\R^n$) by $\overline M$. If $M$ is measurable we also define
\begin{enumerate}[(a)]
 \item[(d)] $\out M=\outsupp \chi_M,$
\end{enumerate}
where $\chi_M$ is the characteristic function of $M$.
\end{definition}

\begin{lemma}\label{lemma:support_properties}
For every measurable function $\kappa:\Omega \to \R$ and 
every measurable set $M$ the following properties hold.
\begin{enumerate}[(a)]
\item $\supp \kappa, \outsupp \kappa, \out M\subseteq \overline\Omega$ are closed.
\item $\innsupp \kappa\subseteq \Omega$ is open.
\item $\innsupp \kappa\subseteq \supp \kappa \subseteq \outsupp \kappa$.
\item $\out (\supp \kappa)=\outsupp \kappa$ 
\item If $\supp \kappa\subseteq \Omega$ and $\Omega\setminus \supp \kappa$ is connected then
$\supp \kappa=\outsupp \kappa$.
\item If $\kappa$ is piecewise continuous then $\supp\kappa=\overline{\innsupp \kappa}$.
\label{lemma:support_properties_pcw_cont}
\end{enumerate}
\end{lemma}
\begin{proof}
\begin{enumerate}[(a)]
\item[(a)] and (b) immediately follow from definition \ref{def:support}.

\item[(c)]
If $\kappa=0$ (a.e.) on a relatively open set $U\subset \overline\Omega$, then 
$\kappa=0$ (a.e.) on the open set $U\cap \innsupp \kappa$. From the definition of the
inner support, it follows that $U\cap \innsupp \kappa=\emptyset$. This shows the first inclusion in (c).
The second inclusion is obvious.
 
\item[(d)] follows from the fact that for every relatively open set $U\subseteq \overline\Omega$ we have
\begin{align*}
\kappa=0 \mbox{ (a.e.) on $U$} & \quad \mbox{ if and only if } \quad U\subseteq \overline \Omega \setminus \supp\kappa\\
& \quad \mbox{ if and only if } \quad \chi_{\supp \kappa}=0 \mbox{ (a.e.) on $U$.}
\end{align*}

\item[(e)] Since $\supp\kappa \subseteq \Omega$ implies that $\overline \Omega\setminus \supp\kappa$ contains $\partial \Omega$,
(e) immediately follows from (c) and  definition \ref{def:support}.

\item[(f)] Let $\kappa$ be continuous on an open set $O\subset \Omega$ where $\Omega\setminus O$ has zero measure.
The assertion follows from (a) and (c) if we can show that for every $x\in \overline \Omega$
\[
x\notin \overline{\innsupp \kappa}\quad \mbox{ implies } \quad x\notin \supp \kappa.
\] 
Let $x\notin \overline{\innsupp \kappa}$. Then there exists a relatively open set $B\subset \overline\Omega$ with $x\in B$ and $B\cap \overline{\innsupp \kappa}=\emptyset$. Obviously, $\{ \xi\in O: \kappa(\xi)\neq 0\}\subseteq \innsupp \kappa$, so that 
$\kappa=0$ on $O\cap B$. Since $\Omega\setminus O$ has zero measure, we have that $\kappa=0$ (a.e.) on $B$ and thus
$B\cap \supp\kappa=\emptyset$ which shows the assertion.

%are obvious. (c) is simple
%(c): If $\kappa=0$ a.e. on an open $U$, then $U$ cannot intersect innsupp, so first inclusion follows.
\end{enumerate}
\end{proof}

As a consequence of Lemma \ref{lemma:support_properties}(e) and (f) we have 
\begin{corollary}\label{cor:exakt}
If $\kappa$ is piecewise continuous, $\supp \kappa\subseteq \Omega$ and $\Omega\setminus \supp \kappa$ is connected then 
\[
\overline{\innsupp \kappa}=\supp \kappa=\outsupp \kappa.
\]
\end{corollary}

\section{Monotonicity and localized potentials}\label{sect:monotonicity_locpot}

\subsection{A monotonicity principle}

Our main theoretical tools are a monotonicity estimate and the theory of localized potentials.
The following estimate goes back to Ikehata, Kang, Seo, and Sheen \cite{Kan97,ikehata1998size}, cf., also the similar results in Ide et al.~\cite{Ide07}, Kirsch \cite{Kir02}, and in \cite{HS09,HS10}.
For the convenience of the reader we state the estimate together with a short proof that we copy from \cite[lemma~2.1]{HS10}.

\begin{lemma}\label{le:monotonicity}
Let $\sigma_1,\sigma_2\in L^\infty_+(\Omega)$ be two conductivities, $g\in L^2_\diamond(\Omega)$ be an applied boundary current and $u_2:=u^g_{\sigma_2}\in H^1_\diamond(\Omega)$. Then
\begin{equation}\label{eq:monotonicity_inequality}
 \int_\Omega(\sigma_1-\sigma_2) |\nabla u_2|^2 \dx 
\geq \Langle g, \left(\Lambda(\sigma_2)-\Lambda(\sigma_1)\right)g \Rangle
\geq \int_\Omega\frac{\sigma_2}{\sigma_1}(\sigma_1-\sigma_2)|\nabla u_2|^2\dx.
\end{equation}
\end{lemma}
\begin{proof}
Let $u_1:=u^g_{\sigma_1}\in H^1_\diamond(\Omega)$. From \eqref{eq:math_model_variationel} we deduce
\begin{equation*}%\label{eq:conclusion_of_math_model_va_1}
\int_\Omega \sigma_1\nabla u_1\cdot\nabla u_2 \dx
= \Langle g, \Lambda(\sigma_2) g\Rangle 
= \int_\Omega \sigma_2\nabla u_2\cdot\nabla u_2 \dx
\end{equation*}
and thus 
\begin{align*}
 \int_\Omega\sigma_1|\nabla(u_1-u_2)|^2\,\mathrm{d}x
&=\int_\Omega\sigma_1|\nabla u_1|^2\,\mathrm{d}x-2\int_\Omega\sigma_2|\nabla u_2|^2\,\mathrm{d}x+
\int_\Omega\sigma_1|\nabla u_2|^2\,\mathrm{d}x\\
&=\Langle g, \Lambda(\sigma_1) g\Rangle 
- \Langle g, \Lambda(\sigma_2) g\Rangle 
+ \int_\Omega (\sigma_1-\sigma_2)|\nabla u_2|^2\,\mathrm{d}x.
\end{align*}
Since the left hand side is non-negative, the first asserted inequality follows.

Interchanging $\sigma_1$ and $\sigma_2$ we obtain 
\begin{align*}
\lefteqn{\Langle g, \left(\Lambda(\sigma_2)-\Lambda(\sigma_1)\right)g\Rangle}\\
&= \int_\Omega (\sigma_1-\sigma_2)|\nabla u_1|^2\dx
  + \int_\Omega \sigma_2|\nabla (u_2-u_1)|^2\dx\\
&= \int_\Omega \left( \sigma_1 |\nabla u_1|^2 + \sigma_2 |\nabla u_2|^2
- 2 \sigma_2 \nabla u_1\cdot \nabla u_2 \right)\dx\\
&=\int_\Omega\sigma_1\left|\nabla u_1-\frac{\sigma_2}{\sigma_1}\nabla u_2\right|^2
\dx + \int_\Omega \left(\sigma_2-\frac{{\sigma_2}^2}{\sigma_1}\right)|\nabla u_2|^2\dx.
\end{align*}
Since the first integral on the right hand-side is non-negative, the second asserted inequality follows.
\end{proof}

We call lemma \ref{le:monotonicity} a \emph{monotonicity estimate} because of the following
corollary. 
\begin{corollary}\label{Cor:monotonicity}
For two conductivities $\sigma_1,\sigma_2\in L^\infty_+(\Omega)$
\begin{equation}\label{eq:monotonicity_for_conductivitys}
\sigma_1\leq\sigma_2 \quad \mbox{ implies } \quad \Lambda(\sigma_1)\geq\Lambda(\sigma_2).
\end{equation}
%where the first inequality means that $\sigma_1(x)\leq \sigma_2(x)$ for all $x\in \Omega$ %a.e., and the second inequality is to be interpreted in the sense of quadratic forms, i.e.
%\[
%\Langle g, \left(\Lambda(\sigma_1)-\Lambda(\sigma_2)\right) g\Rangle \geq 0.
%\]
\end{corollary}

\begin{remark}
Corollary \ref{Cor:monotonicity} already yields a simple monotonicity based reconstruction algorithm.
Assume that the conductivity in the investigated object is $\sigma=1+\chi_D$, 
where the measurable set $D\subseteq \Omega$ describes the unknown inclusion. 
Then for all other measurable sets $B\subseteq \Omega$
\begin{equation}\label{eq:monotonicity_for_sets}
 B\subseteq D\quad \mbox{ implies } \quad \Lambda(1+\chi_B)\geq\Lambda(\sigma),
\end{equation}
so that the set
\[
R:=\bigcup \left\{ B\subseteq\Omega\,:\, B \mbox{ measurable, and } 
 \Lambda(1+\chi_B)\geq\Lambda(\sigma) \right\}
\]
is an upper bound of $D$.

A numerical approximation of (this upper bound of) $D$ can be
calculated by choosing a number of small balls $B=B_\epsilon(z)\subseteq \Omega$ (with center $z\in \Omega$
and radius $\epsilon>0$) and marking all balls where the \emph{monotonicity test} $\Lambda(1+\chi_B)\geq\Lambda(\sigma)$ holds true.
Algorithms based on this idea have been worked out and numerically
tested in the works of Tamburrino and Rubinacci \cite{Tamburrino02,Tamburrino06}.
\end{remark}

Also, Lemma~\ref{le:monotonicity} gives an estimate for the Fr\'echet derivative of $\Lambda$ that 
will be the basis for linearizing our monotonicity tests without losing shape information (cf.\ \cite{HS10} for the origin of this idea).
\begin{corollary}\label{Cor:monotonicity_linearized}
Let $\sigma\in L^\infty_+(\Omega)$. Let $\Lambda(1)$ be the NtD-Operator corresponding to the background conductivity $1$ and $\Lambda'(1)$ be its Fr\'echet derivative (see subsection \ref{subsect:NtD}). Then
\begin{equation*}
 %\int_\Omega(\sigma-1) |\nabla u_2|^2 \dx 
\Lambda'(1)(1-\sigma) 
\geq  \Lambda(1)-\Lambda(\sigma) 
\geq  \Lambda'(1)\left( \frac{1}{\sigma}(1-\sigma)\right). 
%\int_\Omega\frac{1}{\sigma}(\sigma-1)|\nabla u_2|^2\dx.
\end{equation*}
\end{corollary}

Of course, in practical EIT applications, it is not possible to measure 
boundary data with infinite precision. Moreover, with a limited number of electrodes on the boundary of an imaging subject (and limited accuracy),
we can only obtain a finite-dimensional approximation to the true NtD. Also, we can only calculate finite-dimensional approximations
of the NtD for test conductivities (and their linearized counterparts).
Hence, let us comment on the stability of monotonicity tests with respect to such errors.

\begin{remark}\label{rem:convergence}
Monotonicity/definiteness tests can be stably implemented in the following sense. Let $A\in \mathcal L(H)$ be a selfadjoint compact operator on a Hilbert space $H$, and let $(A^\delta)_{\delta>0}\subseteq \mathcal L(H)$ be a family of compact 
(e.g.\ finite dimensional) approximations with
\[
\norm{A^\delta-A}_\mathcal {L(H)}<\delta.
\]
Possibly replacing $A^\delta$ by its symmetric part, we can assume that $A^\delta$ is selfadjoint.

For $\alpha>0$, we define the regularized definiteness test
\[
R_\alpha(A^\delta):= \left\{ \begin{array}{l l} 1 & \mbox{ if $\langle A^\delta g,g\rangle\geq -\alpha\norm{g}^2$ for all $g\in H$,}\\
0 & \mbox{ otherwise,}
\end{array}\right.
\]
which is equivalent to checking whether the smallest eigenvalue of $A^\delta$ is not below $-\alpha$.

If $A\geq 0$ then $\langle A^\delta g,g\rangle\geq -\delta\norm{g}^2$ for all $g\in H$.
If $A\not \geq 0$ then $A$ has a negative eigenvalue $\lambda<0$ so that 
$\langle A^\delta g,g\rangle\geq -\delta\norm{g}^2$ cannot hold for all $g\in H$ for $\delta<|\lambda|/2$.
Hence, 
\[
R_{\delta}(A^\delta) = \left\{ \begin{array}{l l} 1 & \mbox{ if $A\geq 0$},\\
0 & \mbox{ if $A\not\geq 0$ and $\delta$ is sufficiently small}.
\end{array}\right.
\]
\end{remark}

\subsection{Localized potentials}

We will show that a certain converse of
the monotonicity relation \req{monotonicity_for_conductivitys}, resp., \req{monotonicity_for_sets}
holds true. The main theoretical tool for this result is to use the theory of localized 
potentials by one of the authors \cite{Geb08} to control the energy terms $|\nabla u|^2$ in the monotonicity estimate in lemma~\ref{le:monotonicity}.

Roughly speaking, \cite{Geb08} shows that there exist electric potentials which have 
arbitrarily large energy $|\nabla u|^2$ in some region and arbitrarily small energy in
another region, as long as the high-energy region can be reached from the boundary without crossing the low-energy region.
%Figure?

We will make use of the following variant of the result in \cite{Geb08}.
%We will make use of the following formulation of this result.

%%%%%%%%%%%%%%%%%%%%%%%%%%%%%%%%%%%%%%%%%%%%%%%loc_special_start%%%%%%%%%%%%%%%%%%%%%%%%%%%%%%%%%%%%%%%%%%%%%%%%%
\begin{theorem}\label{thm:loc_pot}
Let $D_1,D_2\subseteq\overline\Omega$ be two measurable sets with 
\[
\inn D_1 \nsubseteq \out D_2.
\]
Furthermore let $\sigma\in L^\infty_+(\Omega)$ be piecewise analytic.

Then there exists $(g_m)_{m\in\mathbb{N}}\subset L^2_\diamond(\partial \Omega)$ such that the 
solutions $(u_m)_{m\in \mathbb{N}}\subset \Hd^1(\Omega)$ of
\[
\nabla \cdot \sigma \nabla u_m=0 \quad \mbox{ in $\Omega$}, \qquad \sigma \partial_\nu u_m|_{\partial \Omega}=g_m, 
\]
fulfill
\[
 \lim_{m\rightarrow\infty}\int_{D_1}|\nabla u_m|^2\,\mathrm{d}x=\infty\quad\text{and}\quad\lim_{m\rightarrow\infty}\int_{D_2}|\nabla u_m|^2\,\mathrm{d}x=0.
\]
\end{theorem}
\begin{proof}
The proof is a slight adaptation of the one in \cite[Sect.~2.2]{Geb08}, see also 
\cite{Har09sim} for the general approach.
% Since $\partial \Omega$ is a set of zero measure,
%integrals with domain $D_j$ agree with that over $D_j\cap \Omega$.
\begin{enumerate}
\item[(a)] \emph{Reformulation as range (non-)inclusion}

We define the \emph{virtual measurement} operators $L_j$ ($j=1,2$) by
\[
L_j:\; L^2(D_j)^n \to  L_\diamond^2(\partial \Omega), \qquad F\mapsto u|_{\partial \Omega},
\]
where $u\in H_\diamond^1(\Omega)$ solves 
\begin{equation*}
\int_\Omega \sigma \nabla u \cdot \nabla w \dx =\int_{D_j} F\cdot \nabla w \dx 
\quad \mbox{ for all } w\in H_\diamond^1(\Omega).
\end{equation*}
Note that this implies $\nabla\cdot \sigma \nabla u=0$ in $\Omega \setminus \overline D_j$
and if $\overline D_j\subseteq \Omega$ then it also implies the homogeneous Neumann boundary
condition $\sigma \partial_\nu u|_{\partial \Omega}=0$.

It is easily checked that the dual operators
\[
L_j':\; L_\diamond^2(\partial \Omega) \to L^2(D_j)^n, \quad j=1,2
\]
are given by 
$L_j' g=\nabla v|_{D_j}$ where $v\in H_\diamond^1(B)$ solves
\begin{equation*} 
\nabla \cdot \sigma \nabla v=0 \quad \mbox{ and } \quad 
\sigma \partial_\nu v|_{\partial \Omega}=g.
\end{equation*}

Now the assertion is equivalent to the statement
\[
\nexists C>0 \ : \ \norm{L_1' g}\leq C \norm{L_2' g}\quad \forall g\in L^2_\diamond(\partial \Omega)
\]
which is (see e.g. \cite[Lemma 2.5]{Geb08}) equivalent to the range (non-)inclusion
\begin{equation}\labeq{LocPot_range_incl}
\range(L_1)\nsubseteq \range(L_2).
\end{equation}

\item[(b)] \emph{Proof of the range (non-)inclusion \req{LocPot_range_incl}}

Since $\inn D_1 \nsubseteq \out D_2$,
the set $\inn D_1$ must intersect one of the sets $U$ in the definition of the outer support 
of $\chi_{D_2}$. Hence, there exists a set $U\subset \overline \Omega\setminus \out D_2$ with $U$ (relatively) open in $\overline \Omega$, 
$U$ connected to $\partial \Omega$, and $U\cap D_1$ contains an open ball $B$. Possibly shrinking the ball we can assume that
$\overline B\subset \Omega$ and that $(U\cap \Omega)\setminus \overline B$ is connected.
%\footnote{$U$ connected $B_{2\epsilon}\subseteq U$ $\Longrightarrow U\setminus \overline B_\epsilon$ connected: %Proof in den Appendix?}

Let $L_B$ denote the virtual measurement operator corresponding to the ball $B$. Obviously, $B\subseteq D_1$ implies $\range(L_B)\subseteq \range(L_1)$, so that it suffices to prove that 
\begin{equation}\labeq{LocPot_range_incl2}
\range(L_B) \nsubseteq \range(L_2).
\end{equation}
To that end let $\varphi\in \range(L_B)\cap \range(L_2)$. Then there exist $u_B, u_2\in \Hd^1(\Omega)$
with $u_B|_{\partial \Omega}=\varphi=u_2|_{\partial \Omega}$, $\sigma \partial_\nu u_B|_{U\cap \partial \Omega}=0=\sigma \partial_\nu u_2|_{U\cap\partial \Omega}$, and
\begin{eqnarray*}
\nabla \cdot \sigma \nabla u_B&=&0 \quad \mbox{ in } \Omega\setminus \overline B,\\
\nabla \cdot \sigma \nabla u_2&=&0 \quad \mbox{ in } U.
\end{eqnarray*}
By unique continuation $u_B=u_2$ in $U\setminus \overline B$. Hence 
\[
u:=\left\{ \begin{array}{l l} u_B & \mbox{ in } \Omega\setminus \overline B,\\
u_2 & \mbox{ in } \overline B \end{array}\right.
\]
defines a function $u\in \Hd^1(\Omega)$ with $\nabla \cdot \sigma \nabla u=0$ in $\Omega$ and homogeneous Neumann boundary data $\sigma \partial_\nu u|_{\partial \Omega}=0$. It follows that $\varphi=u|_{\partial \Omega}=0$ and thus we have shown that $\range(L_B) \cap  \range(L_2)=\{ 0 \}$.

Finally, using unique continuation again, we obtain that $L_B'$ is injective, so that $\range(L_B)$ is dense in $\Ld^2(\partial \Omega)$. A fortiori, $\range(L_B)\neq \{ 0 \}$, which, together with $\range(L_B) \cap  \range(L_2)=\{ 0 \}$, proves \req{LocPot_range_incl2} and thus the assertion.
\end{enumerate}
\end{proof}

Note that theorem~\ref{thm:loc_pot} also holds for less regular conductivities as long as a unique continuation property is fulfilled, 
and that localized potentials can be constructed by solving regularized operator equations, cf.\ \cite{Geb08}.

We now show that (regardless of regularity) 
the properties of the localized potentials do not depend on the conductivity in the low energy region:
\begin{lemma}\label{lemma:loc_pot_remark}
Let $D_1,D_2\subseteq\overline\Omega$ be two measurable sets. 
Let $\sigma,\tau\in L^\infty_+(\Omega)$ and $u_m^\sigma,u_m^\tau\in \Hd^1(\Omega)$ denote the corresponing 
solutions of 
\begin{align*}
\nabla \cdot \sigma \nabla u^\sigma_m=0 & \quad \mbox{ in $\Omega$}, \qquad \sigma \partial_\nu u^\sigma_m|_{\partial \Omega}=g_m,\\ 
\nabla \cdot \tau \nabla u^\tau_m=0 & \quad \mbox{ in $\Omega$}, \qquad \tau \partial_\nu u^\tau_m|_{\partial \Omega}=g_m,
\end{align*}
for a sequence of boundary currents $(g_m)_{m\in\mathbb{N}}\subset L^2_\diamond(\partial \Omega)$.

If $\supp (\sigma-\tau)\subseteq D_2$ then 
\[
 \lim_{m\rightarrow\infty}\int_{D_1}|\nabla u^\sigma_m|^2\,\mathrm{d}x=\infty\quad\text{and}\quad\lim_{m\rightarrow\infty}\int_{D_2}|\nabla u^\sigma_m|^2\,\mathrm{d}x=0
\]
holds if and only if
\[
\lim_{m\rightarrow\infty}\int_{D_1}|\nabla u^\tau_m|^2\,\mathrm{d}x=\infty\quad\text{and}\quad\lim_{m\rightarrow\infty}\int_{D_2}|\nabla u^\tau_m|^2\,\mathrm{d}x=0.
\]
\end{lemma}
\begin{proof}
For both conductivities, $\sigma$ and $\tau$, we define the virtual measurement operators
\[
L_2^\sigma,L_2^\tau:\ L^2(D_2)^n \to  L_\diamond^2(\partial \Omega), 
\]
as in the proof of theorem \ref{thm:loc_pot}. If $u^\sigma|_{\partial \Omega} = L_2^\sigma F$ with $F\in L^2(D_2)$ and a solution $u^\sigma\in H_\diamond^1(\Omega)$
of 
\begin{equation*}
\int_\Omega \sigma \nabla u^\sigma \cdot \nabla w \dx =\int_{D_2} F\cdot \nabla w \dx 
\quad \mbox{ for all } w\in H_\diamond^1(\Omega).
\end{equation*}
then $u^\sigma$ also solves
\begin{equation*}
\int_\Omega \tau \nabla u^\sigma \cdot \nabla w \dx =\int_{D_2} \left( F + (\tau - \sigma) \nabla u^\sigma \right)\cdot \nabla w \dx 
\quad \mbox{ for all } w\in H_\diamond^1(\Omega).
\end{equation*}
This shows that $\range(L_2^\sigma)\subseteq \range(L_2^\tau)$. As in the proof of theorem \ref{thm:loc_pot}, this implies that
\[
\exists C>0:\ \int_{D_2}|\nabla u^\sigma_m|^2\mathrm{d}x \leq C \int_{D_2}|\nabla u^\tau_m|^2\,\mathrm{d}x \quad \mbox{ for all $m\in \N$.}
\]
By interchanging $\sigma$ and $\tau$, we obtain that
\[
\lim_{m\rightarrow\infty}\int_{D_2}|\nabla u^\sigma_m|^2\,\mathrm{d}x=0 \quad \Longleftrightarrow \quad
\lim_{m\rightarrow\infty}\int_{D_2}|\nabla u^\tau_m|^2\,\mathrm{d}x=0.
\]
Using the same argument on $D_1\cup D_2$ it follows that also 
\[
\lim_{m\rightarrow\infty}\int_{D_1\cup D_2}|\nabla u^\sigma_m|^2\,\mathrm{d}x=\infty \quad \Longleftrightarrow \quad
\lim_{m\rightarrow\infty}\int_{D_1\cup D_2}|\nabla u^\tau_m|^2\,\mathrm{d}x=\infty,
\]
so that the assertion follows.
\end{proof}

\begin{remark}
Localized potentials can be numerically constructed by solving regularized operator equations (see \cite{Geb08}), and they can be used to probe for an unknown inclusion in the spirit of 
the probe or needle method, cf.\ e.g. \cite{ikehata2005new,ikehata2007probe}. We briefly sketch the idea on a simple test example. Assume that the conductivity is $\sigma=1+\chi_D$ and that $(g_m)_{m\in \N}$ is a sequence such that the 
solutions $(u_m)_{m\in \mathbb{N}}\subset \Hd^1(\Omega)$ of 
$\Delta u_m=0$ and $\partial_\nu u_m|_{\partial \Omega}=g_m$ fulfill
\[
 \lim_{m\rightarrow\infty}\int_{D_1}|\nabla u_m|^2\,\mathrm{d}x=\infty\quad\text{and}\quad\lim_{m\rightarrow\infty}\int_{D_2}|\nabla u_m|^2\,\mathrm{d}x=0.
\]
Then the monotonicity estimate in lemma \ref{le:monotonicity} yields that
\begin{alignat*}{3}
D & \subseteq D_2 & & \quad \text{ implies } \quad & \left| \langle g_m, (\Lambda(1)-\Lambda(\sigma)) g_m\rangle \right| & \to 0,\\
D_1 & \subseteq D & & \quad \text{ implies } \quad  & \left| \langle g_m, (\Lambda(1)-\Lambda(\sigma)) g_m\rangle \right| & \to \infty.
\end{alignat*}
Choosing $D_2$ to cover most of $\Omega$ and $D_1$ to be, e.g., a small ball inside $\Omega\setminus D_2$, one may thus estimate the shape of $D$ by slowly shrinking $D_2$. 

Such an algorithm would however suffer from high computational cost (to construct a high number of localized potentials) and it is not clear how to check the limit of 
$\langle g_m, (\Lambda(1)-\Lambda(\sigma)) g_m\rangle$ in a numerically stable way.
Furthermore, the choice of the sets $D_1$ and $D_2$ would certainly impose some geometrical restrictions on the shapes of inclusions that can be recovered. 

In the following, we take a different approach. The monotonicity methods derived in the next section do not require the numerical construction of localized potentials. We will only require the above abstract existence results for localized potentials 
%(theorem~\ref{thm:loc_pot} and lemma~\ref{lemma:loc_pot_remark}) 
in order to show that simple monotonicity tests recover the true (outer) shape
of an inclusion. 
\end{remark}

\section{Monotonicity based shape reconstruction}\label{sect:reconstruction}

\subsection{The definite case}\label{subsect:definite}

We will now show how the shape reconstruction problem can be solved via simple monotonicity tests. 
We start with the definite case, in which the inclusions conductivity is everywhere higher or everywhere lower than the background. 
We treat this case separately since it allows a particularly simple reconstruction strategy. Given a small ball 
the following theorems show how to check whether the ball belongs to the inclusion or not.
The proofs of the theorems  are postponed until the end of this subsection. The main idea of this subsection has previously been summarized in the extended conference abstract \cite{HU10}.

\begin{theorem}\label{th:main_th1}
Let $\sigma\in L^\infty_+(\Omega)$ and $\sigma\geq 1$. 

For every open ball $B:=B_\epsilon(z)$ and every $\alpha>0$,
\begin{alignat*}{2}
\alpha \chi_{B}&\leq \sigma-1 \quad & \mbox{ implies } \quad \Lambda(1+\alpha \chi_B)\geq \Lambda(\sigma),\\
B&\nsubseteq \outsupp (\sigma-1)
\quad & \mbox{ implies } \quad 
\Lambda(1+\alpha \chi_{B})\ngeq \Lambda(\sigma).
\end{alignat*}

Hence, the set
\[
 R:=\bigcup_{\alpha>0} \left\lbrace B=B_\epsilon(z)\subseteq\Omega\ : \ \Lambda\left(1+\alpha\chi_{B}\right)\geq\Lambda\left(\sigma\right)\right\rbrace
\]
fulfills
\[
\innsupp (\sigma-1)\subseteq R\subseteq\outsupp (\sigma-1).
\]
\end{theorem}

\begin{example}\label{ex:main_th1}
Let $\sigma=1+\chi_D$ where the inclusion $D$ is open, and $\overline D\subset \Omega$ has a connected complement.
Then for every open ball $B\subseteq \Omega$
\[
B\subseteq D \quad \mbox{ if and only if } \quad \Lambda(1+\chi_B)\geq \Lambda(\sigma).
\]
\end{example}

Note that implementing the monotonicity tests in theorem~\ref{th:main_th1} or example~\ref{ex:main_th1} would be computationally expensive
since for each ball $B$ (and possibly also for each test level $\alpha$) we would have to 
solve the EIT equation with a new inhomogeneous conductivity in order to calculate $\Lambda(1+\alpha \chi_B)$.
The following theorem shows that we can replace the tests by linearized versions, that
do not require such inhomogeneous forward solutions. Since this is a bit counterintuitive, let us stress that the following result is not affected
by the linearization error, no matter how large that may be. The linearized inverse problem in EIT still contains the exact shape information, cf.\ \cite{HS10}.

\begin{theorem}\label{th:main_th1_lin}
Let $\sigma\in L^\infty_+(\Omega)$ and $\sigma\geq 1$. 

For every open ball $B:=B_\epsilon(z)$ and every $\alpha>0$,
\begin{alignat*}{2}
\alpha \chi_{B}&\leq \frac{1}{\sigma}(\sigma - 1) \quad & \mbox{ implies } \quad \Lambda(1)+\alpha \Lambda'(1) \chi_B\geq \Lambda(\sigma),\\
B&\nsubseteq \outsupp (\sigma-1)
\quad & \mbox{ implies } \quad 
\Lambda(1)+\alpha \Lambda'(1) \chi_B\ngeq \Lambda(\sigma).
\end{alignat*}

Hence, the set
\[
 R:=\bigcup_{\alpha>0} \left\lbrace B=B_\epsilon(z)\subseteq\Omega\,:\,
\Lambda(1)+\alpha \Lambda'(1) \chi_B\geq \Lambda(\sigma)\right\rbrace
\]
fulfills
\[
\innsupp (\sigma-1)\subseteq R\subseteq\outsupp (\sigma-1).
\]
\end{theorem}

\begin{example}\label{ex:main_th1_linearized}
Let $\sigma=1+\chi_D$ where the inclusion $D$ is open, and $\overline D\subset \Omega$ has a connected complement.
Then for every ball $B=B_\epsilon(z)$
\[
B\subseteq D \quad \mbox{ if and only if } \quad \Lambda(1)+{\textstyle \frac{1}{2}} \Lambda'(1) \chi_B\geq \Lambda(\sigma).
\]
\end{example}

\emph{Proof of theorem~\ref{th:main_th1}.} 
Let $\sigma\in L^\infty_+(\Omega)$, $\sigma\geq 1$. Let $B=B_\epsilon(z)$ and $\alpha>0$.
Corollary~\ref{Cor:monotonicity} yields that 
\begin{alignat*}{2}
\alpha \chi_{B}\leq \sigma-1 \quad & \mbox{ implies } & \quad \Lambda(1+\alpha \chi_B)\geq \Lambda(\sigma).
\intertext{It remains to show that}
B\nsubseteq \outsupp (\sigma-1)
\quad & \mbox{ implies } & \quad 
\Lambda(1+\alpha \chi_{B})\ngeq \Lambda(\sigma).
\end{alignat*}
Let $B\nsubseteq \outsupp (\sigma-1)$. 
Corollary~\ref{Cor:monotonicity} yields that shrinking 
the ball $B$ only makes $\Lambda(1+\alpha \chi_{B})$ larger, so that we can assume w.l.o.g.\ that
\[
B\subseteq \Omega\setminus  \outsupp (\sigma-1).
\]

We have that $1+\alpha \chi_B$ is piecewise analytic,
\[
B=\inn B \quad \mbox{ and } \quad \outsupp (\sigma-1)=\out (\supp (\sigma-1))
\]
(see lemma \ref{lemma:support_properties}(d)). Hence, we can apply theorem \ref{thm:loc_pot} and obtain a sequence of currents
$(g_m)_{m\in\mathbb{N}}\subset L^2_\diamond(\partial \Omega)$ so that the 
solutions $(u_m)_{m\in\mathbb{N}}\subset \Hd^1(\Omega)$ of
\[
\nabla \cdot (1+\alpha \chi_B) \nabla u_m=0 \quad \mbox{ in $\Omega$}, \qquad (1+\alpha \chi_B) \partial_\nu u_m|_{\partial \Omega}=g_m, 
\]
fulfill
\[
 \lim_{m\rightarrow\infty}\int_{B}|\nabla u_m|^2\,\mathrm{d}x=\infty\quad\text{and}\quad\lim_{m\rightarrow\infty}\int_{\supp (\sigma-1)}|\nabla u_m|^2\,\mathrm{d}x=0.
\]
From lemma~\ref{le:monotonicity} it follows that
\begin{align*}
%\lefteqn{
\Langle g_m, \left(\Lambda(1+\alpha \chi_B)-\Lambda(\sigma)\right)g_m \Rangle
& \leq \int_\Omega (\sigma-1-\alpha \chi_B)|\nabla u_m|^2\dx\\
& = -\alpha \int_B |\nabla u_m|^2\dx
+ \int_{\supp (\sigma-1)} (\sigma-1)|\nabla u_m|^2\dx\\
& \to -\infty,
\end{align*}
and hence $\Lambda(1+\alpha \chi_B)\ngeq \Lambda(\sigma)$.\hfill $\Box$

%%%%%%%%%%%

\emph{Proof of Theorem~\ref{th:main_th1_lin}.} 
Let $\sigma\in L^\infty_+(\Omega)$, $\sigma\geq 1$. Let $B=B_\epsilon(z)$ and $\alpha>0$.

For every $g\in L^2_\diamond(\partial \Omega)$ and solution $u\in \Hd^1(\Omega)$ of
\[
\Delta u=0 \quad \mbox{ in $\Omega$}, \qquad \partial_\nu u|_{\partial \Omega}=g, 
\]
we obtain from lemma~\ref{le:monotonicity}
\begin{align*}
\Langle g, \left( \Lambda(1)+\alpha \Lambda'(1) \chi_B - \Lambda(\sigma) \right)g\Rangle
\geq \int_\Omega \left(\frac{1}{\sigma}(\sigma-1) - \alpha \chi_B\right) |\nabla u|^2\dx.
\end{align*}
This shows that
\begin{alignat*}{2}
\alpha \chi_{B}\leq \frac{1}{\sigma}(\sigma-1) \quad & \mbox{ implies } & \quad 
\Lambda(1)+\alpha \Lambda'(1) \chi_B \geq  \Lambda(\sigma).
\intertext{It remains to show that}
B\nsubseteq \outsupp (\sigma-1)
\quad & \mbox{ implies } & \quad 
\Lambda(1)+\alpha \Lambda'(1) \chi_B \ngeq  \Lambda(\sigma).
\end{alignat*}
To show this let $B\nsubseteq \outsupp (\sigma-1)$. 
The linearized monotonicity relation \req{Frechet_monotonicity} yields that shrinking 
the ball $B$ only makes $\Lambda(1)+\alpha \Lambda'(1) \chi_{B}$ larger, so that we can assume w.l.o.g.\ that
$B\subseteq \Omega\setminus  \outsupp (\sigma-1)$. Then, 
\begin{align*}
\lefteqn{\Langle g, \left( \Lambda(1)+\alpha \Lambda'(1) \chi_B - \Lambda(\sigma) \right)g\Rangle}\\
&\leq \int_\Omega \left( \sigma-1-\alpha \chi_B \right) |\nabla u|^2 \dx%\\
 = -\alpha \int_B |\nabla u|^2\dx
 + \int_{\supp (\sigma-1)} (\sigma-1) |\nabla u|^2\dx,
\end{align*}
so that the assertion follows using localized potentials for the 
background conductivity $1$ and the same sets as in theorem \ref{thm:loc_pot}.\hfill $\Box$

\begin{remark}\label{rem:def_othercase}
If $\sigma\in L^\infty_+(\Omega)$ and $\sigma\leq 1$ then we obtain with the same arguments that for 
every open ball $B\subseteq \Omega$ and every $0<\alpha<1$,
\begin{alignat*}{2}
\alpha \chi_{B}&\leq 1-\sigma    \quad & \mbox{ implies } & \quad \Lambda(1-\alpha \chi_B)\leq \Lambda(\sigma),\\
B&\nsubseteq \outsupp (\sigma-1) \quad & \mbox{ implies } & \quad \Lambda(1-\alpha \chi_{B})\nleq \Lambda(\sigma),
\intertext{and for every open ball $B\subseteq \Omega$ and every $\alpha>0$}
\alpha \chi_{B}&\leq 1-\sigma    \quad & \mbox{ implies } & \quad \Lambda(1)-\alpha \Lambda'(1) \chi_B\leq \Lambda(\sigma),\\
B&\nsubseteq \outsupp (\sigma-1) \quad & \mbox{ implies } & \quad \Lambda(1)-\alpha \Lambda'(1) \chi_B\nleq \Lambda(\sigma).
\end{alignat*}
%Note: the linearized assertions follow from
%\begin{equation}\label{eq:monotonicity_inequality}
%\Langle g, \left(\Lambda(1)-\alpha\Lambda'(1)\chi_B-\Lambda(\sigma)\right)g \Rangle \leq \int_\Omega(\sigma-1+\alpha \chi_B) |\nabla u|^2 \dx 
%\end{equation}
\end{remark}

\begin{remark}
An inspection of the proofs shows that the balls can be replaced by
arbitrary measurable sets $B$ with non-empty interior in theorem~\ref{th:main_th1_lin} (and the second part of
remark \ref{rem:def_othercase}). For theorem~\ref{th:main_th1} (and the first part of remark~\ref{rem:def_othercase}) the sets $B$ 
must additionally possess a piecewise smooth boundary (so that $1+\alpha\chi_B$ remains piecewise analytic). We comment on further generalizations
in section \ref{subsect:extensions}.
\end{remark}

\subsection{The indefinite case}

We now consider the general indefinite case where $\sigma$ is no longer required to be everywhere larger or everywhere smaller
than the background conductivity $1$. Instead of testing whether a small test region is part of the unknown 
inclusion, we will now test whether a large test region contains the unknown inclusions.

The main idea is the following. Consider a large test region $C$ with connected complement. If $C$ overlaps the inclusions
then a large enough, resp., small enough test conductivity on $C$ will make the corresponding test NtD smaller, resp., larger then the measured NtD.
Hence if $C$ overlaps the inclusions then two monotonicity tests (one with a large and one with a small test level on $C$) hold true.
On the other hand, if $C$ does not overlap the inclusions then we can connect the non-overlapped part with the boundary,
and construct a localized potential with large energy in the non-overlapped part and small energy in $C$. 
Depending on whether the conductivity is larger, resp., smaller than the background in the non-overlapped part, 
this localized potential shows that one of the monotonicity tests cannot hold true.

However, for this argument we need a \emph{local definiteness property}. 
If a conductivity differs from the background then there must either be a neighborhood of the boundary where it differs from the background in the positive direction, or a neighborhood where it differs in the negative direction. Note that even $C^\infty$-conductivities might oscillate infinitely 
and thus violate this property. This property holds, however, if the conductivity is either piecewise analytic 
or if the higher-conductivity and lower-conductivity parts have some distance from each other, and the inner support does not
deviate too much from the true support (which already holds, e.g., for piecewise continuous functions, see corollary \ref{cor:exakt}).

More precisely, we assume that $\sigma\in L^\infty_+(\Omega)$ is either piecewise-analytic, or  
\begin{equation}\labeq{indef_separate}
\supp (\sigma-1)^+ \,\cap\, \supp (\sigma-1)^- =\emptyset,\quad  \overline{\innsupp (\sigma-1)}=\supp (\sigma-1)
\end{equation}
where $(\sigma-1)^+:=\max \{\sigma-1,0\}, \quad (\sigma-1)^-:=\min \{\sigma-1,0\}$.

\begin{theorem}\label{th:main_th_disjoint}
Let $\sigma\in L^\infty_+(\Omega)$ either be piecewise-analytic or fufill \req{indef_separate}.

Then, for every set $C\subseteq \overline\Omega$ with $C=\out C$ and every $\alpha>1$ 
\begin{alignat*}{2}
1 -  {\TS \frac{1}{\alpha}} \chi_{C}&\leq \sigma \quad & \mbox{ implies } & \quad \Lambda(1 - {\TS \frac{1}{\alpha}} \chi_{C})\geq \Lambda(\sigma),\\
1 + \alpha \chi_{C}&\geq \sigma \quad & \mbox{ implies } & \quad \Lambda(1+\alpha \chi_C)\leq \Lambda(\sigma)
\end{alignat*}
and
\[ 
\Lambda(1+\alpha \chi_C) \leq \Lambda(\sigma)\leq \Lambda(1 - {\TS \frac{1}{\alpha}} \chi_{C})
\quad \mbox{ implies } \quad \outsupp (\sigma-1)\subseteq C.
\]

Hence, 
\begin{align*}
 R& := \bigcap \{ C=\out C\subseteq \overline\Omega,\ 
%\\  & \qquad \qquad 
\exists \alpha>1 :\ 
\Lambda(1+\alpha \chi_C) \leq \Lambda(\sigma)\leq \Lambda(1 - {\TS \frac{1}{\alpha}} \chi_{C})
\}
\end{align*}
fulfills
$
R=\outsupp (\sigma-1).
$
\end{theorem}

We postpone the proof until the end of this subsection 
and first give an example and formulate the linearized version.

\begin{example}\label{ex:main_th_disjoint}
Let $\sigma=1+\chi_{D^+}-\frac{1}{2}\chi_{D^-}$ where $D^+,D^-\subseteq \Omega$ are open sets with
$\overline{D^+}\cap \overline{D^-}=\emptyset$, and $\overline{D^+}\cup \overline{D^-}\subset \Omega$ has a connected complement.

Then for every closed set $C\subset \Omega$ with connected complement $\Omega\setminus C$
\[
D^+ \cup D^-\subseteq C \quad \mbox{ if and only if } \quad \Lambda(1+\chi_C)\leq \Lambda(\sigma)\leq \Lambda(1-{\TS \frac{1}{2}}\chi_C).
\]
\end{example}

\begin{theorem}\label{th:main_th_disjoint_linearized}
Under the assumptions of theorem~\ref{th:main_th_disjoint} we have that for every set $C\subseteq \overline\Omega$ with $C=\out C$ and every $\alpha>0$ 
\begin{alignat*}{2}
1 - \alpha \chi_{C}&\leq  2 - {\TS \frac{1}{\sigma}} \quad & \mbox{ implies } & \quad \Lambda(1)-\alpha \Lambda'(1)\chi_C  \geq \Lambda(\sigma),\\
1 + \alpha \chi_{C}&\geq \sigma \quad & \mbox{ implies } & \quad \Lambda(1)+\alpha \Lambda'(1)\chi_C\leq \Lambda(\sigma)
\end{alignat*}
and
\[ 
\Lambda(1)+\alpha \Lambda'(1)\chi_C \leq \Lambda(\sigma) \leq \Lambda(1)-\alpha \Lambda'(1)\chi_C 
\quad \mbox{ implies } \quad \outsupp (\sigma-1)\subseteq C.
\]

Hence, 
\begin{align*}
 R& := \bigcap \{ C=\out C\subseteq \overline\Omega,\ 
\exists \alpha>0 :\\
 & \qquad \qquad \quad \Lambda(1)+\alpha \Lambda'(1)\chi_C \leq \Lambda(\sigma) \leq \Lambda(1)-\alpha \Lambda'(1)\chi_C \}
\end{align*}
fulfills
$
R=\outsupp (\sigma-1).
$
\end{theorem}

\begin{example}\label{ex:main_th_disjoint_linearized}
Let $\sigma=1+\chi_{D^+}-\frac{1}{2}\chi_{D^-}$ where $D^+,D^-\subseteq \Omega$ are open sets with
$\overline{D^+}\cap \overline{D^-}=\emptyset$, and $\overline{D^+}\cup \overline{D^-}\subset \Omega$ has a connected complement.

Then for every closed set $C\subset \Omega$ with connected complement $\Omega\setminus C$
\[
D^+ \cup D^-\subseteq C \quad \mbox{ if and only if } \quad \Lambda(1)+ \Lambda'(1)\chi_C\leq \Lambda(\sigma)\leq \Lambda(1)- \Lambda'(1)\chi_C.
\]
\end{example}

%%%%
\emph{Proof of Theorem~\ref{th:main_th_disjoint}.} 
Let $\alpha>1$ and $C=\out C\subseteq \overline\Omega$. Then $C$ is closed and thus measurable, 
so that $1-{\TS \frac{1}{\alpha}} \chi_C, 1+\alpha \chi_C\in L^\infty_+(\Omega)$.

Corollary~\ref{Cor:monotonicity} yields the first two assertions
\begin{alignat*}{2}
1 -  {\TS \frac{1}{\alpha}} \chi_{C}&\leq \sigma \quad & \mbox{ implies } & \quad \Lambda(1 - {\TS \frac{1}{\alpha}} \chi_{C})\geq \Lambda(\sigma),\\
1 + \alpha \chi_{C}&\geq \sigma \quad & \mbox{ implies } & \quad \Lambda(1+\alpha \chi_C)\leq \Lambda(\sigma)
\end{alignat*}

It remains to show that $\outsupp (\sigma-1)\nsubseteq C$
implies that either
\[ 
\Lambda(1 - {\TS \frac{1}{\alpha}} \chi_{C})\ngeq \Lambda(\sigma)\quad\text{ or }\quad \Lambda(1+\alpha \chi_C)\nleq \Lambda(\sigma).
\]
Let $\outsupp (\sigma-1)\nsubseteq C=\out C$. Then there exists a relatively open set
$U\subseteq \overline{\Omega}$ that is connected to $\partial \Omega$ where
$\sigma|_U\not\equiv 1$ and $C\cap U=\emptyset$.

We first prove the assertion for the case that $\sigma$ is piecewise analytic. Using the local definiteness property  derived in corollary~\ref{cor:pcw_anal_def}) in the appendix, we can choose (note that $\Omega \setminus D_2\subseteq U$ implies $C\subseteq D_2$)
\[
D_1,D_2\subseteq \overline\Omega,\quad \mbox{ with } \quad
D_1=\inn D_1\not\subseteq \out D_2=D_2, \quad C\subseteq D_2,
\]
so that either
\begin{enumerate}[(a)]
\item $\DS \sigma\geq 1 \mbox{ on } \Omega \setminus D_2, \quad \sigma-1\in L^\infty_+(D_1)$, or
\item $\DS \sigma\leq 1 \mbox{ on } \Omega \setminus D_2, \quad 1-\sigma\in L^\infty_+(D_1)$.
\end{enumerate}
Replacing $D_1$ with $D_1\setminus  \out D_2$, we can also assume that $D_1\cap D_2=\emptyset$.

We then use the localized potentials theorem \ref{thm:loc_pot} for the homogeneous conductivity $\tau=1$
and obtain a sequence $(g_m)_{m\in\mathbb{N}}\subset L^2_\diamond(\partial \Omega)$ so that the 
solutions $(u^\tau_m)_{m\in \mathbb{N}}\subseteq \Hd^1(\Omega)$ of
\[
\nabla \cdot \tau \nabla u^\tau_m=0 \quad \mbox{ in $\Omega$}, \qquad \tau \partial_\nu u^\tau_m|_{\partial \Omega}=g_m, 
\]
fulfill 
\[
 \lim_{m\rightarrow\infty}\int_{D_1}|\nabla u_m|^2\,\mathrm{d}x=\infty\quad\text{and}\quad\lim_{m\rightarrow\infty}\int_{D_2}|\nabla u_m|^2\,\mathrm{d}x=0.
\]
Since $C\subseteq D_2$ it follows from lemma~\ref{lemma:loc_pot_remark} that the solutions $u^\tau_m$ for the conductivities $\tau=1-{\TS \frac{1}{\alpha}} \chi_C$
and $\tau=1+\alpha \chi_C$ have the same property.

Hence, in case (a), we apply lemma~\ref{le:monotonicity} with $\tau=1+\alpha \chi_C$ and obtain (using that $\sigma\geq 1$ on $\Omega\setminus (D_1\cup D_2)$, and that $C\subseteq D_2$)
\begin{align*}
\lefteqn{\Langle g_m, \left(\Lambda(1 + \alpha \chi_C)-\Lambda(\sigma)\right)g_m \Rangle}\\
&\geq \int_\Omega \frac{1 + \alpha \chi_C}{\sigma} (\sigma- (1 + \alpha \chi_C)) |\nabla u^\tau_m|^2 \dx\\
& = \int_{\Omega\setminus (D_1\cup D_2)} \frac{\sigma-1}{\sigma} |\nabla u^\tau_m|^2 \dx + \int_{D_1} \frac{\sigma-1}{\sigma} |\nabla u^\tau_m|^2 \dx\\
& \qquad {} + \int_{D_2} \frac{1 + \alpha \chi_C}{\sigma} (\sigma- (1 + \alpha \chi_C)) |\nabla u^\tau_m|^2 \dx\\ 
&\geq \int_{D_1} \frac{\sigma-1 }{\sigma}  |\nabla u^\tau_m|^2 \dx
+ \int_{D_2} \frac{1 + \alpha \chi_C}{\sigma} (\sigma- (1 + \alpha \chi_C)) |\nabla u^\tau_m|^2 \dx\to \infty.
\intertext{In case (b), we apply lemma~\ref{le:monotonicity} with $\tau=1-{\TS \frac{1}{\alpha}} \chi_C$ and obtain (using that $\sigma\leq 1$ on $\Omega\setminus (D_1\cup D_2)$, and that $C\subseteq D_2$)}
\lefteqn{\Langle g_m, \left(\Lambda(1-{\TS \frac{1}{\alpha}} \chi_C)-\Lambda(\sigma)\right)g_m \Rangle}\\
&\leq \int_\Omega (\sigma- ( 1-{\TS \frac{1}{\alpha}}\chi_C )) |\nabla u^\tau_m|^2\dx\\
&= \int_{\Omega\setminus (D_1\cup D_2)} (\sigma-1) |\nabla u^\tau_m|^2\dx 
+ \int_{D_1} (\sigma- 1) |\nabla u^\tau_m|^2\dx\\
& {} \qquad + \int_{D_2} (\sigma- ( 1-{\TS \frac{1}{\alpha}}\chi_C )) |\nabla u^\tau_m|^2\dx\\
&\leq \int_{D_1} (\sigma-1) |\nabla u^\tau_m|^2\dx + \int_{D_2} (\sigma- ( 1-{\TS \frac{1}{\alpha}}\chi_C )) |\nabla u^\tau_m|^2\dx \to -\infty,
\end{align*}
which proves the assertion for piecewise analytic conductivities.

Now we prove that the assertion also holds for (not necessary piecewise analytic) conductivities fulfilling \req{indef_separate}.
It suffices to show that also in this case, there exist
\[
D_1,D_2\subseteq \overline\Omega,\quad \mbox{ with } \quad
D_1=\inn D_1\not\subseteq \out D_2=D_2, \quad C\subseteq D_2,
\]
such that either (a) or (b) from above holds.

First note that if $\supp (\sigma-1)^+$ and $\supp (\sigma-1)^-$ are disjoint compact sets, then
\[
\delta:=\mathrm{dist}\left( \supp (\sigma-1)^+, \supp (\sigma-1)^- \right)>0.
\]

$\sigma|_U\not\equiv 1$ implies that there exists a point $y\in U\cap \supp(\sigma-1)$.
Let $x\in \partial \Omega\cap U$. Since $\partial \Omega$ is a smooth boundary and
$U\cap \Omega$ is open and connected, we can connect $x$ and $y$ with a continuous path
\[
\gamma:\ [0,1]\to U, \quad \gamma(0)=x,\ \gamma(1)=y.
\]
Using that $U$ is relatively open, there exists, for each $t\in [0,1]$, a ball $B_t:=B_{\epsilon(t)}(\gamma(t))$
with radius $\epsilon(t)<\delta/2$ and $B_t\cap \overline\Omega \subseteq U$.

By compactness of $\gamma([0,1])$ we can choose a finite number
$0\leq t_1<\ldots<t_N\leq 1$, so that
\[
\gamma([0,1])\subset \left( B_{t_1}\cup\ldots \cup B_{t_N} \right) \cap \overline{\Omega}.
\]
Since $\gamma(1)=y\in \supp(\sigma-1)$, there exists a smallest index $J$
for which 
\[
B_{t_J}\cap \overline{\innsupp (\sigma-1)}=B_{t_J}\cap \supp(\sigma-1)\neq \emptyset
\] 
so that there exists an open set $D_1\subseteq B_{t_J}$ with $|\sigma-1|\in L^\infty_+(D_1)$.

We define $D_2:=\Omega \setminus \left( B_{t_1}\cup\ldots \cup B_{t_J} \right)$. Then
\[
D_1,D_2\subseteq \overline\Omega,\quad \mbox{ with } \quad
D_1=\inn D_1\not\subseteq \out D_2=D_2, \quad C\subseteq D_2.
\]
Furthermore, since $B_{t_J}$ has diameter less than $\delta$,
it can not intersect both $\supp (\sigma-1)^+$ and
$\supp (\sigma-1)^-$, so that either 
\begin{enumerate}[(a)]
\item $\DS \sigma\geq 1 \mbox{ on } \Omega \setminus D_2, \quad \sigma-1\in L^\infty_+(D_1)$, or
\item $\DS \sigma\leq 1 \mbox{ on } \Omega \setminus D_2, \quad 1-\sigma\in L^\infty_+(D_1)$,
\end{enumerate}
which finishes the proof.
\hfill $\Box$

%%%%%%%%%%%%%%%%%%%

\emph{Proof of Theorem~\ref{th:main_th_disjoint_linearized}.} 

If $1-\alpha\chi_C\leq 2- \frac{1}{\sigma}$, then
$\alpha\chi_C\geq \frac{1}{\sigma}(1-\sigma)$, so that \req{Frechet_monotonicity} %implies $\alpha\Lambda'(1)\chi_C\leq \Lambda'(1)\left(\frac{1}{\sigma}(1-\sigma)\right)$ and we obtain from 
and corollary \ref{Cor:monotonicity_linearized} imply that
\begin{align*}
\Lambda(\sigma)\leq \Lambda(1)-\Lambda'(1)\left(\frac{1}{\sigma}(1-\sigma)\right)
\leq \Lambda(1)-\alpha\Lambda'(1)\chi_C.
\end{align*}
Likewise, if $1+\alpha\chi_C\geq \sigma$ then \req{Frechet_monotonicity} %implies $\alpha \Lambda'(1)\chi_C\leq \Lambda'(1)(\sigma-1)$ and we obtain from 
and corollary \ref{Cor:monotonicity_linearized} imply that
\begin{align*}
\Lambda(\sigma)\geq \Lambda(1)-\Lambda'(1)(1-\sigma)\geq \Lambda(1)+\alpha\Lambda'(1)\chi_C.
\end{align*}
This shows the first two assertions.

% To show the third assertion let $\outsupp (\sigma-1)\nsubseteq C$. As in the 
% proof of theorem~\ref{th:main_th_disjoint} we can then obtain sets 
% \[
% D_1,D_2\subseteq \overline \Omega \quad \mbox{ with } \quad D_1=\inn D_1 \not\subseteq \out D_2=D_2\supseteq C,
% \]
% where either
% \begin{enumerate}[(a)]
% \item $\DS \sigma\geq 1 \mbox{ on } \Omega \setminus D_2, \quad \sigma-1\in L^\infty_+(D_1)$, or
% \item $\DS \sigma\leq 1 \mbox{ on } \Omega \setminus D_2, \quad 1-\sigma\in L^\infty_+(D_1)$.
% \end{enumerate}

Moreover, lemma \ref{le:monotonicity} yields that for all $\alpha\in \R$ 
\begin{align*}
\Langle (\Lambda(1)+\alpha\Lambda'(1)\chi_C-\Lambda(\sigma) ) g,g\Rangle 
&\geq \int_\Omega \left( \frac{1}{\sigma}(\sigma-1) -\alpha \chi_C \right) |\nabla u_g|^2\dx
\end{align*}
and
\begin{align*}
\Langle (\Lambda(1)-\alpha\Lambda'(1)\chi_C-\Lambda(\sigma) ) g,g\Rangle 
&\leq \int_\Omega(\sigma-1+\alpha\chi_C) |\nabla u_g|^2 \dx,
\end{align*}
where $u_g\in \Hd^1(\Omega)$ solves
$\Delta u_g=0$ and $\partial_\nu u_g|_{\partial \Omega}=g$. Hence, the third assertion follows
by using localized potentials for the homogeneous conductivity and the same sets $D_1$, $D_2$ as in theorem~\ref{th:main_th_disjoint}.
\hfill $\Box$

\subsection{Remarks and extensions}\label{subsect:extensions}
Let us comment on some extensions and generalizations of our results. 
Our assumption that the background conductivity is equal to $1$ and that we are given measurements on the complete boundary $\partial \Omega$
have been merely for the ease of presentation. All our results and proofs remain valid if 
$\partial \Omega$ is replaced by an arbitrarily small
open piece $S\subset \partial \Omega$ and we are given the partial Neumann-to-Dirichlet operator 
\[
\Lambda(\sigma):\ L^2_\diamond(S)\to L^2_\diamond(S), \quad g\mapsto u^g_\sigma|_{S},
\]
where $u^g_\sigma\in H^1_\diamond(\Omega)$ is the unique solution of
\begin{equation*}
\nabla\cdot\sigma\nabla u^g_\sigma = 0 \text{ in }\Omega,\quad\sigma\partial_\nu u_\sigma^g\vert_{\partial \Omega} = \left\{\begin{array}{l l} g & \text{ on } S,\\
0  & \text{ on } \partial \Omega\setminus S.\end{array}\right.
\end{equation*}
Also, all the results still hold when the background conductivity $1$ is replaced by a known piecewise analytic function.

Let us also note that our results require piecewise analyticity for only two purposes: the existence of localized potentials
and the local definiteness property. Localized potentials exist for less regular conductivities, it only requires that
the solutions of the corresponding elliptic EIT equations satisfy a unique continuation property, cf.\ \cite{Geb08}.
Local definiteness can hold for quite general functions, if additional assumption are made (e.g., 
that positive and negative part are separated as in \req{indef_separate}).
However, the authors are not aware of any natural function classes beyond piecewise-analytic functions in which a property in the spirit of 
theorem \ref{thm:pcw_anal_def} holds without further assumptions.

\begin{appendix}
%\section{Appendix}
%\subsection{Local definiteness of piecewise analytic functions}
\section{Local definiteness of piecewise analytic functions}

In this appendix, we show that piecewise analytic functions have a \emph{local definiteness property}. If they do not vanish identically
then there is either a neighborhood of the boundary where they differ from zero in the positive direction, or a neighborhood where they
differ in the negative direction.

The property follows from the arguments used in the proofs of \cite[Theorem~4.2]{Har09} and \cite[Lemma~3.7]{HS10}. 
However, some subtle and not entirely trivial topological details were omitted in \cite{Har09,HS10}, which is why we
give the proof here in full detail.

\begin{theorem}\label{thm:pcw_anal_def}
Let $\Omega\subset \R^n$, $n\geq 2$ be a smoothly bounded domain, and let $\sigma\in L_+^\infty(\Omega)$ be piecewise analytic.
Let $U\subseteq \overline\Omega$ be relatively open and connected to $\partial \Omega$, and let $\sigma|_U\not\equiv 0$.

Then we can find a subset $V\subseteq U$ with the same properties,
on which $\sigma$ does not change sign, i.e.
\begin{enumerate}[(a)]
\item $V\subseteq \overline\Omega$ is relatively open, $V$ is connected to $\partial \Omega$, $V\subseteq U$,
\item $\sigma|_V\not\equiv 0$, and either $\sigma|_V\geq 0$ or $\sigma|_V\leq 0$.
\end{enumerate}
\end{theorem}

Obviously, if a piecewise analytic function is not identically zero, we can find a neighborhood where it is bounded away from zero.
Hence, choosing $D_2:=\Omega\setminus V$, we obtain the following corollary.
%which is useful in view of the localized potentials of theorem \ref{thm:loc_pot}.

\begin{corollary}\label{cor:pcw_anal_def}
%Let $\Omega\subset \R^n$, $n\geq 2$ be a smoothly bounded domain, and $\sigma\in L_+^\infty(\Omega)$ be piecewise analytic.
%For every $U\subseteq \overline\Omega$ that is relatively open and connected to $\partial \Omega$.
%
Under the assumptions of theorem \ref{thm:pcw_anal_def}
we can choose 
\[
D_1,D_2\subseteq \overline\Omega,\quad \mbox{ with } \quad
D_1=\inn D_1\not\subseteq \out D_2=D_2, \quad \Omega \setminus D_2\subseteq U,
\]
and either
\begin{alignat*}{2}
\sigma|_{\Omega\setminus D_2}\geq 0, & \quad & \sigma|_{D_1}&\in L_+^\infty(D_1),
\quad \mbox{ or }\\
\sigma|_{\Omega\setminus D_2}\leq 0, & \quad & -\sigma|_{D_1}&\in L_+^\infty(D_1).
\end{alignat*}
\end{corollary}

In the following, let $\Omega\subset \R^n$, $n\geq 2$ be a smoothly bounded domain, and $\sigma\in L_+^\infty(\Omega)$ be piecewise analytic
with respect to
\[
\overline{\Omega}= \overline{O_1\cup \ldots \cup O_M},\quad \partial O_m=\bigcup_{k\in \N} \overline{\Gamma_m^k}
\]
where, w.l.o.g, we assume that every $\partial O_m$ consists of infinitely many pieces.
Furthemore, let $U\subseteq \overline\Omega$ be relatively open and connected to $\partial \Omega$.

\begin{lemma}\label{lemma:pcw_anal_first_bndry}
There exists an open ball $B\subseteq \R^n$ such that 
\[
B\cap \overline \Omega\subseteq U, \quad \mbox{ and } \quad B\cap \overline \Omega\mbox{ is connected to $\partial \Omega$},
\]
and for one of the $O_m$ and one of its smooth boundary pieces $\Gamma_m^k\subseteq \partial O_m$,
\[
B\cap \Omega=B\cap O_m \quad \mbox{ and } \quad B\cap \partial \Omega\subseteq \Gamma_m^k.
\]
\end{lemma}
\begin{proof}
Since $U$ is relatively open and $U\cap \partial \Omega\neq \emptyset$ , there exists an open ball $B$ with
$\emptyset \neq B\cap \partial \Omega$, $B\cap \overline \Omega\subseteq U$, and by shrinking $B$ we can assume that
\[
\emptyset\neq S:=\overline B \cap \partial \Omega \subseteq U.
\]

$\overline \Omega=\overline{O_1\cup \ldots \cup O_M}$ implies that
\[
\partial \Omega\subseteq \bigcup_{m=1}^M \partial O_m=\bigcup_{m=1}^M\bigcup_{k\in \N} \overline{\Gamma_m^k}
\quad \mbox{ and thus } \quad
S=\bigcup_{m,k} \overline{\Gamma_m^k}\cap S.
\]
%(Proof: $x\in \partial \Omega$. Then $x\not\in \Omega$, so $x\not\in O_j$ but $x\in \overline %O_j$)
By Baire's theorem, one of the countably many closed sets $\overline{\Gamma_m^k}\cap S$ must have 
non-empty interior in $S$. Hence, for one of the $\overline{\Gamma_m^k}$, there exists an open ball 
$B$ with
\[
\emptyset\neq B\cap \partial \Omega \subseteq \overline{\Gamma_m^k}\cap U.
\]
%Since $\Omega$ is smoothly bounded, we can achieve by shrinking $B$ that also $B\cap \partial \Omega$ is a non-empty smooth boundary piece. 

Moreover, $B\cap \partial \Omega$ must intersect $\Gamma_m^k$ because of the following dimension theoretical argument,
cf., e.g., the classical book of Hurewicz and Wallman \cite[Ch. IV, \S 4]{hurewicz1948dimension}.
$\Omega\cap B$ is an open (neither empty nor dense) subset of the 
$(n-1)$-dimensional ball $B$. As a subset of $B$, the boundary of $\Omega\cap B$ is
$\partial \Omega\cap B$, which shows that $\partial \Omega\cap B$ is $(n-1)$-dimensional
(and not of lesser dimension).
$\Gamma_m^k$ is a (neither empty nor dense) open subset of a set that is homeomorphic to $\R^{n-1}$. Hence, $\Gamma_m^k$ is 
$(n-1)$-dimensional and $\overline{\Gamma_m^k}\setminus \Gamma_m^k$ is $(n-2)$-dimensional. 
This shows that $B\cap \partial \Omega\not\subseteq \overline{\Gamma_m^k}\setminus \Gamma_m^k$, so that by shrinking $B$ we can assume that 
\[
\emptyset\neq B\cap \partial \Omega \subseteq \Gamma_m^k\cap U.
\]
Finally, we can shrink $B$ so that $B\cap \Omega=B\cap O_m$ and that $B\cap \Omega$ is connected.
\end{proof}

\begin{lemma}\label{lemma:pcw_anal_definiteness}
Every open ball $B\subseteq \R^n$ that intersects a smooth boundary piece $\Gamma_m^k$ contains
an open subball $B'\subseteq B$ intersecting $\Gamma_m^k$,
% $B'\cap \Gamma_m^k\neq \emptyset$ 
where either
\[
\sigma|_{B'\cap O_m}\geq 0 \quad \mbox{ or } \quad \sigma|_{B'\cap O_m}\leq 0.
\]
%open subdomains $D,B'\subseteq B$ with $\overline D\subseteq B'\cap O_m$, so
%that either
%\[
%\sigma|_{B'\cap O_m}\geq 1,\ \sigma|_{D}-1 \in L^\infty_+(D)
%\quad \mbox{ or } \quad
%\sigma|_{B'\cap O_m}\leq 1,\ 1-\sigma|_{D} \in L^\infty_+(D)
%\]
%note that $B'\cap O_m$ is connected if B' is small enough
\end{lemma}
\begin{proof}
We use an argument of Kohn and Vogelius \cite{Koh84}.
If $\sigma|_{B\cap O_m}\equiv 0$ then the assumption is trivial.
Otherwise, by analyticity, there must be a smallest $k\in \N$, so that the normal derivative
$\partial^k_{\nu(z)} \sigma(z)$ is not identically zero for all $z\in \Gamma_m^k\cap B$.
Hence there is a neighbourhood of a point $z\in \Gamma_m^k\cap B$ 
on which either $\sigma\geq 1$ or $\sigma\leq 1$.
\end{proof}

Now we are ready to prove the local definiteness property.

\emph{Proof of theorem \ref{thm:pcw_anal_def}.} From lemma~\ref{lemma:pcw_anal_first_bndry}
we obtain an open ball $B\subseteq \R^n$ with
\[
B\cap \overline \Omega\subseteq U, \quad B\cap \overline \Omega\mbox{ is connected to $\partial \Omega$},
\]
and (w.l.o.g.)
\[
B\cap \Omega=B\cap O_1, \quad B\cap \partial \Omega\subseteq \Gamma_1^1.
\]
If $\sigma$ is not identically zero on $O_1$ then the assertion follows from lemma~\ref{lemma:pcw_anal_definiteness}.

Otherwise, $M>1$, and the set $V:=B\cap \overline{\Omega}$ has the following properties: 
\begin{enumerate}[(i)]
\item $V$ is a relatively open subset of $\overline{\Omega}$ that is connected
to $\partial \Omega$, \label{thm_pcw_anal_prop1}
\item $V$ fulfills $B\cap \overline{\Omega}\subseteq V\subseteq U$, \label{thm_pcw_anal_prop2}
\item $\sigma|_{V}=0$. \label{thm_pcw_anal_prop3}
\end{enumerate}
Obviously these properties are closed under union, so that we can choose $V$ to be the maximal set fulfilling (\ref{thm_pcw_anal_prop1})--(\ref{thm_pcw_anal_prop3}).

Now we show that 
\begin{equation}\labeq{thm_pcw_anal_hilf}
\emptyset \neq \partial V\cap U \cap \Omega \subseteq \bigcup_{m=1}^M \partial O_m.
\end{equation}
Since $V$ is relatively open in $\Omega$ and $V\subseteq U$ it follows that
$V\cap \Omega$ is relatively open in $U\cap \Omega$. If $\partial V$ has no intersection with
$U\cap \Omega$, then $V\cap \Omega$ is relatively closed, so that $U\cap \Omega=V\cap \Omega$,
which contradicts $\sigma|_U\not\equiv 0$. Hence, the $\partial V\cap U \cap \Omega\neq \emptyset$.
To show the second assertion in \req{thm_pcw_anal_hilf}, assume that there exists $z\in \partial V\cap \Omega\cap U$ with $z\in O_m$ for one $O_m$. 
Then we can choose an open ball $B\subseteq O_m\cap U$ containing $z$. Since $\sigma$ is analytic on $B$ and
$B\cap V$ has non-empty interior, it follows that $\sigma|_B\equiv 0$, and hence $V\cup B$ has the properties (\ref{thm_pcw_anal_prop1})--(\ref{thm_pcw_anal_prop3}). This contradicts the
maximality of $V$, so that also the second assertion in \req{thm_pcw_anal_hilf} must hold.

Because of \req{thm_pcw_anal_hilf} we can choose an open ball $B$ with $\overline{B}\subseteq U\cap \Omega$ and
\begin{equation}\labeq{thm_pcw_anal_hilf2}
\emptyset\neq S:=\overline B\cap \partial V=\bigcup_{m,k} \overline{\Gamma_m^k}\cap S.
\end{equation}
Using the same arguments as in the proof of lemma~\ref{lemma:pcw_anal_first_bndry}
it follows that by shrinking $B$ we can assume that
\[
\emptyset \neq B\cap \partial V\subseteq \Gamma_m^k
\]
with a smooth boundary piece $\Gamma_m^k$ of one $O_m$. 

Since $\partial O_m\subseteq \bigcup_{m'\neq m}\partial O_ {m'}\cup \partial \Omega$, equation \req{thm_pcw_anal_hilf2} still holds if
we restrict the union to all $m'\in \{1,\ldots,M\}\setminus \{m\}$. By repeating the above argument
(and possibly shrinking $B$ again) we obtain $m'\neq m$ with
\[
\emptyset \neq \overline B\cap \partial V\subseteq \Gamma_m^k\cap \Gamma_{m'}^{k'}
\]
From the definition of smooth boundary pieces it follows that (if we choose $B$ small enough)
\[
B\subseteq \overline{(B\cap O_m)\cup {B\cap O_{m'}}},
\]
so that either $B\cap O_m$ or $B\cap O_{m'}$, but not both, intersects $V$. W.l.o.g, let 
$B\cap O_m$ intersect $V$. Then $\sigma|_{O_m}\equiv 0$, and using lemma~\ref{lemma:pcw_anal_definiteness}
we can shrink $B$ so that $\sigma|_{B\cap O_{m'}}$ is either non-negative or non-positive.
Hence $B\cup V$ fulfills the above properties (\ref{thm_pcw_anal_prop1}) and (\ref{thm_pcw_anal_prop2}) and it is a proper superset of $V$. Hence,
$\sigma$ cannot identically vanish on $B\cup V$, which shows that $B\cup V$ 
fulfills the assertion of theorem \ref{thm:pcw_anal_def}.
\hfill $\Box$

% Using the analyticity of $\sigma$ on $O_m$ we obtain the following corollary.
% \begin{corollary}\label{pcw_anal_cor_def}
% Let $z\in \Gamma_m^k$ be a point on a smooth boundary piece of one of the $O_m$.
% If $\sigma|_{O_m}$ is not identically one, then every ball $B\ni z$ contains a subball $B'\ni z$ and an open subdomain $D$
% with $\overline D\subseteq B'\cap O_m$, so
% that either
% \[
% \sigma|_{B'\cap O_m}\geq 1,\ \sigma|_{D}-1 \in L^\infty_+(D)
% \quad \mbox{ or } \quad
% \sigma|_{B'\cap O_m}\leq 1,\ 1-\sigma|_{D} \in L^\infty_+(D)
% \]
% \end{corollary}

\end{appendix}

%%%%%%%

\bibliography{literaturliste}
\bibliographystyle{abbrv}

\end{document}